\documentclass[11pt]{article}

\usepackage[T1]{fontenc}
\usepackage{lmodern}
\usepackage[margin=1in]{geometry}
\usepackage{amsmath,amssymb,amsthm}
\usepackage{mathtools}
\usepackage{microtype}
\usepackage{enumitem}
\usepackage{aliascnt}
\usepackage{tikz}
\usepackage[hidelinks]{hyperref}
\usepackage[nameinlink,capitalize,noabbrev]{cleveref}

\newtheorem{theorem}{Theorem}[section]

\newaliascnt{lemma}{theorem}
\newtheorem{lemma}[lemma]{Lemma}
\aliascntresetthe{lemma}

\newaliascnt{proposition}{theorem}
\newtheorem{proposition}[proposition]{Proposition}
\aliascntresetthe{proposition}

\newaliascnt{corollary}{theorem}
\newtheorem{corollary}[corollary]{Corollary}
\aliascntresetthe{corollary}

\theoremstyle{definition}
\newaliascnt{definition}{theorem}
\newtheorem{definition}[definition]{Definition}
\aliascntresetthe{definition}

\newaliascnt{example}{theorem}
\newtheorem{example}[example]{Example}
\aliascntresetthe{example}

\theoremstyle{remark}
\newaliascnt{remark}{theorem}

\aliascntresetthe{remark}

\newaliascnt{question}{theorem}
\newtheorem{question}[question]{Question}
\aliascntresetthe{question}

\newcommand{\B}{\mathcal{B}}
\newcommand{\dcl}[1]{\operatorname{cl}_{\diamond,#1}}
\newcommand{\E}{E}
\newcommand{\botstar}{S_{\mathrm{bot}}}
\newcommand{\topstar}{S_{\mathrm{top}}}

\title{Cographs and Minimum Diamond-Generating Edge Sets in Boolean Lattices}
\author{Mahesh Ramani\thanks{Email: \href{mailto:mahesh.ramani.iyer@gmail.com}{\texttt{mahesh.ramani.iyer@gmail.com}}}}
\date{}

\hypersetup{
  pdftitle={Cographs and Minimum Diamond-Generating Edge Sets in Boolean Lattices},
  pdfauthor={Mahesh Ramani},
  pdfsubject={Combinatorics and graph theory},
  pdfkeywords={Boolean lattice, diamond closure, cograph, cover graph, graph twins, generating set}
}

\begin{document}

\maketitle

\begin{abstract}
We study a local closure operation on the cover edges of a Boolean lattice:
whenever the two lower edges or the two upper edges of a square face are
present, all four edges of that square are added. We prove that every set of
cover edges generating the full cover graph of \(\B_n\) has cardinality at
least \(n\), and we classify all generators attaining this bound. For a graph
\(G\) on \([n]\), let
\[
S_G=\{N_G(i)\to N_G(i)\cup\{i\}:i\in[n]\}.
\]
Then \(S_G\) diamond-generates the full cover graph if and only if \(G\) is a
cograph, and every minimum-cardinality generator arises uniquely in this way.
Consequently, labeled minimum diamond-generating sets of \(\B_n\) are in
bijection with labeled cographs on \(n\) vertices.
\end{abstract}


\section{Introduction}

A diamond in the cover graph of a lattice is a square formed by two successive
cover relations in either order. If two adjacent edges on one side of the
square are known, a natural completion rule is to add the two edges on the
other side. This diamond operation appears in the study of rational relations
among pseudo-roots of noncommutative polynomials; in particular, Retakh and
Saks study diamond closure in cover graphs of modular lattices
\cite{RetakhSaks2021}. The present paper isolates a finite extremal question for
the Boolean lattice.

Let \(\B_n\) denote the Boolean lattice of subsets of \([n]\). We ask:
what is the smallest number of cover edges from which every cover edge of
\(\B_n\) can be obtained by repeated diamond completions, and what do all
smallest generating sets look like? The answer is governed by cographs, the
class of graphs generated from single vertices by disjoint union and join
\cite{CorneilLerchsBurlingham1981}. Cographs are also characterized as the
finite graphs with no induced four-vertex path \cite{Seinsche1974}.

For a graph \(G\) on a finite vertex set \(X\), write \(N_G(i)\) for the open
neighborhood of \(i\), and associate to \(G\) the set of cover edges
\[
S_G=\{N_G(i)\to N_G(i)\cup\{i\}:i\in X\}.
\]
Our main result is the following.

\begin{theorem}[Main classification]\label{thm:main}
Let \(X\) be a finite nonempty set.
\begin{enumerate}[label=\textup{(\roman*)}]
  \item Every diamond-generating set of cover edges of \(\B_X\) has
  cardinality at least \(|X|\), and this bound is attained.
  \item For a graph \(G\) on \(X\), the set \(S_G\) diamond-generates
  \(\B_X\) if and only if \(G\) is a cograph.
  \item Every diamond-generating set of cardinality \(|X|\) is equal to
  \(S_G\) for a unique cograph \(G\) on \(X\).
\end{enumerate}
\end{theorem}

The lower bound follows because a diamond completion cannot introduce a
coordinate direction that was absent from the initial set. A generator of
minimum cardinality therefore contains exactly one edge in each direction. A
coordinate-projection lemma then forces a pairwise symmetry among the lower
endpoints of those edges, so they are precisely the open neighborhoods of a
simple graph. Cographs generate by induction through their disjoint-union and
join constructions. Conversely, generation passes to every induced subgraph,
while any seed set that begins to expand must contain two edges forming one
half of a diamond; for neighborhood seeds, such a pair is exactly a pair of
false or true twins. The hereditary twin characterization of cographs then
completes the proof.

Throughout, graphs are finite, simple, and undirected. The word
\emph{minimum} refers to minimum cardinality, rather than inclusion-minimality.

\section{Diamond closure and coordinate projection}

For a finite set \(X\), let \(\B_X=2^X\) be the Boolean lattice ordered by
inclusion. When \(X=[n]=\{1,\dots,n\}\), we also write \(\B_n\).

A cover edge of \(\B_X\) has the form
\[
e_i(A): A\longrightarrow A\cup\{i\},
\qquad A\subseteq X\setminus\{i\}.
\]
The element \(i\) is called the \emph{direction} of the edge. We write
\(\E(\B_X)\) for the full set of cover edges.

A diamond is determined by a set \(A\subseteq X\) and distinct elements
\(i,j\in X\setminus A\). Its lower and upper halves are, respectively,
\[
L(A;i,j)=\{e_i(A),e_j(A)\}
\]
and
\[
U(A;i,j)=\{e_i(A\cup\{j\}),e_j(A\cup\{i\})\}.
\]
Thus its four-edge set is \(L(A;i,j)\cup U(A;i,j)\).

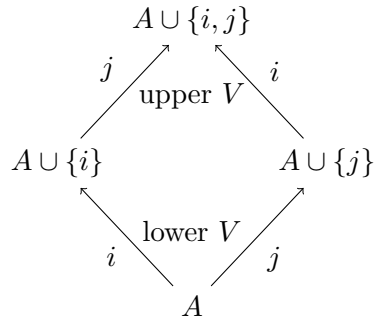
\begin{center}
\begin{tikzpicture}[scale=1.05,baseline=(current bounding box.center)]
  \node (top) at (0,1.8) {$A\cup\{i,j\}$};
  \node (left) at (-1.7,0) {$A\cup\{i\}$};
  \node (right) at (1.7,0) {$A\cup\{j\}$};
  \node (bottom) at (0,-1.8) {$A$};
  \draw[->] (bottom) -- node[below left] {$i$} (left);
  \draw[->] (bottom) -- node[below right] {$j$} (right);
  \draw[->] (left) -- node[above left] {$j$} (top);
  \draw[->] (right) -- node[above right] {$i$} (top);
  \node at (0,-0.85) {lower $V$};
  \node at (0,0.85) {upper $V$};
\end{tikzpicture}
\end{center}

\begin{definition}[Synchronous diamond closure]\label{def:closure}
Let \(S\subseteq \E(\B_X)\). Define an increasing sequence
\(S^{(0)},S^{(1)},\dots\) by \(S^{(0)}=S\) and by letting \(S^{(t+1)}\) be
\(S^{(t)}\) together with all four edges of every diamond for which either
\(L(A;i,j)\subseteq S^{(t)}\) or \(U(A;i,j)\subseteq S^{(t)}\). The
\emph{diamond closure} of \(S\) is
\[
\dcl{X}(S)=\bigcup_{t\ge 0}S^{(t)}.
\]
We call \(S\) \emph{diamond-generating} if
\(\dcl{X}(S)=\E(\B_X)\).
\end{definition}

Because \(\E(\B_X)\) is finite, the sequence in \cref{def:closure} stabilizes.
Its stable value is equivalently the smallest set of cover edges containing
\(S\) that is closed under diamond completion.

We next record the projection property that will be used twice in the proof.
For \(Y\subseteq X\), define the coordinate projection of an edge in a retained
direction by
\[
\pi_Y(e_i(A))=e_i(A\cap Y),\qquad i\in Y,
\]
and discard edges whose direction is not in \(Y\). For a set \(T\) of edges,
let \(\pi_Y(T)\) be the set of all retained projected edges.

\begin{lemma}[Coordinate projection]\label{lem:projection}
For every \(S\subseteq \E(\B_X)\) and every \(Y\subseteq X\),
\[
\pi_Y\bigl(\dcl{X}(S)\bigr)
\subseteq
\dcl{Y}\bigl(\pi_Y(S)\bigr).
\]
Consequently, if \(S\) diamond-generates \(\B_X\), then \(\pi_Y(S)\)
diamond-generates \(\B_Y\).
\end{lemma}

\begin{proof}
Let \(S^{(t)}\) be the closure sequence in \(\B_X\), and let \(T^{(t)}\) be
the closure sequence in \(\B_Y\) starting from \(T^{(0)}=\pi_Y(S)\). We prove
by induction on \(t\) that
\[
\pi_Y(S^{(t)})\subseteq T^{(t)}.
\]
The assertion is immediate for \(t=0\).

Consider a diamond used in passing from \(S^{(t)}\) to \(S^{(t+1)}\), with
directions \(i\) and \(j\). If \(i,j\in Y\), its four edges project to the
four edges of a diamond in \(\B_Y\), and the triggering lower or upper \(V\)
projects to the corresponding triggering \(V\). Hence every projected edge
added by this completion belongs to \(T^{(t+1)}\).

If exactly one of \(i,j\) lies in \(Y\), the two edges in the retained
direction project to the same edge of \(\B_Y\). Completing the diamond
therefore produces no new projected edge: the projection of the newly added
retained-direction edge is already the projection of the retained-direction
edge in the triggering \(V\). If neither direction lies in \(Y\), all four
edges are discarded. These cases establish the induction step.

Taking the union over \(t\) gives the stated inclusion. If \(S\) generates
\(\B_X\), then
\[
\E(\B_Y)
=\pi_Y(\E(\B_X))
=\pi_Y\bigl(\dcl{X}(S)\bigr)
\subseteq \dcl{Y}\bigl(\pi_Y(S)\bigr),
\]
and the reverse inclusion in \(\E(\B_Y)\) is automatic.
\end{proof}

\section{Minimum size and graph encoding}

\begin{lemma}[Direction support]\label{lem:directions}
If no edge of \(S\subseteq \E(\B_X)\) has direction \(i\), then no edge of
\(\dcl{X}(S)\) has direction \(i\).
\end{lemma}

\begin{proof}
Every diamond has two directions, and each of its lower and upper halves
contains one edge in each direction. Thus a diamond completion can produce a
new edge in direction \(i\) only when its triggering half already contains an
edge in direction \(i\). Induction on the closure rounds proves the claim.
\end{proof}

\begin{proposition}[Minimum size]\label{prop:minsize}
Every diamond-generating subset of \(\E(\B_X)\) has cardinality at least
\(|X|\), and this bound is attained.
\end{proposition}

\begin{proof}
By \cref{lem:directions}, a generating set must contain at least one edge in
each of the \(|X|\) directions.

For sharpness, consider the bottom star
\[
\botstar=\{e_i(\varnothing):i\in X\}.
\]
We prove by induction on \(r\) that every edge \(e_i(A)\) with \(|A|\le r\)
lies in \(\dcl{X}(\botstar)\). The case \(r=0\) is the definition of
\(\botstar\). Suppose the assertion holds through \(r\), and let
\(|A|=r+1\). Choose \(j\in A\) and set \(B=A\setminus\{j\}\). Since
\(i\notin A\), the two edges \(e_i(B)\) and \(e_j(B)\) are available by the
induction hypothesis. They form the lower half of a diamond, whose upper half
contains
\[
e_i(B\cup\{j\})=e_i(A).
\]
Thus all cover edges are generated.

Dually, the top star
\[
\topstar=\{e_i(X\setminus\{i\}):i\in X\}
\]
also generates. Hence the minimum cardinality is \(|X|\).
\end{proof}

Let \(S\) now be a diamond-generating set of cardinality \(|X|\). It contains
exactly one edge in each direction, so there are uniquely determined sets
\(A_i\subseteq X\setminus\{i\}\) such that
\[
S=\{e_i(A_i):i\in X\}.
\]

\begin{lemma}[Symmetry lemma]\label{lem:symmetry}
For distinct \(i,j\in X\),
\[
j\in A_i\quad\Longleftrightarrow\quad i\in A_j.
\]
\end{lemma}

\begin{proof}
Project to the two-coordinate Boolean lattice on \(Y=\{i,j\}\). By
\cref{lem:projection}, \(\pi_Y(S)\) generates \(\B_Y\). All seed edges in
directions outside \(Y\) are discarded, so \(\pi_Y(S)\) consists of exactly
one edge in direction \(i\) and one edge in direction \(j\). More precisely,
\[
\pi_Y(e_i(A_i))=
\begin{cases}
 e_i(\varnothing),&j\notin A_i,\\
 e_i(\{j\}),&j\in A_i,
\end{cases}
\]
and analogously for the direction-\(j\) edge.

The lattice \(\B_Y\) has a single diamond. A pair containing one edge in each
direction generates that diamond exactly when the two edges form its lower
half or its upper half. If one is lower and the other is upper, no diamond
completion is available and the closure does not expand. Therefore the two
projected edges must both be lower or both be upper, which is precisely
\[
j\in A_i\quad\Longleftrightarrow\quad i\in A_j.
\]
\end{proof}

\begin{corollary}[Graph encoding]\label{cor:encoding}
Every diamond-generating set of cardinality \(|X|\) is equal to \(S_G\) for a
unique graph \(G\) on \(X\).
\end{corollary}

\begin{proof}
Define adjacency for distinct \(i,j\in X\) by
\[
ij\in \E(G)\quad\Longleftrightarrow\quad j\in A_i.
\]
The relation is symmetric by \cref{lem:symmetry}, and it has no loops because
\(i\notin A_i\). Hence it defines a simple graph. Its open neighborhoods are
\(N_G(i)=A_i\), so \(S=S_G\). The neighborhoods determine every adjacency,
which gives uniqueness.
\end{proof}

\begin{example}\label{ex:stars}
For the empty graph on \(X\), every neighborhood is empty and \(S_G=\botstar\).
For the complete graph on \(X\), every neighborhood is \(X\setminus\{i\}\)
and \(S_G=\topstar\).
\end{example}

\section{Cographs and twins}

The join \(G_1\vee G_2\) is obtained from the disjoint union of \(G_1\) and
\(G_2\) by adding every edge between their vertex sets. A \emph{cograph} is a
graph obtained from one-vertex graphs by repeated disjoint union and join. This
is equivalent to the standard definition as a graph with no induced copy of
the four-vertex path \(P_4\) \cite{CorneilLerchsBurlingham1981,Seinsche1974}.

Two distinct vertices \(x,y\) are \emph{false twins} if they are nonadjacent
and have the same neighbors outside \(\{x,y\}\). They are \emph{true twins}
if they are adjacent and have the same neighbors outside \(\{x,y\}\). Either
type will be called a pair of \emph{twins}.

We use the following standard characterization; a proof is included for
completeness.

\begin{lemma}[Hereditary twin characterization]\label{lem:twin-characterization}
A finite graph \(G\) is a cograph if and only if every induced subgraph of
\(G\) with at least two vertices contains a pair of twins.
\end{lemma}

\begin{proof}
Suppose first that \(G\) is a cograph. We prove the hereditary statement by
induction on \(|V(G)|\). It is immediate for a one-vertex graph. A nontrivial
cograph has a construction whose final operation expresses it either as a
disjoint union
\[
G=G_1\sqcup\cdots\sqcup G_k
\]
or as a join
\[
G=G_1\vee\cdots\vee G_k,
\]
where \(k\ge2\) and each \(G_a\) is a smaller cograph.

Let \(H\) be an induced subgraph of \(G\) with at least two vertices, and put
\(H_a=H[V(H)\cap V(G_a)]\). If some \(H_a\) has at least two vertices, then it
has twins by induction; those vertices remain twins in \(H\), because every
vertex outside \(G_a\) is adjacent to both of them in the join case and to
neither of them in the disjoint-union case. If every nonempty \(H_a\) is a
singleton, choose vertices in two distinct parts. They are false twins in the
disjoint-union case and true twins in the join case.

Conversely, suppose every induced subgraph of \(G\) with at least two vertices
has twins. We induct on \(|V(G)|\). Choose twins \(x,y\) in \(G\) and let
\(H=G-y\). The graph \(H\) has the same hereditary twin property and is
therefore a cograph by induction. Take a construction tree for \(H\). Replace
the leaf corresponding to \(x\) by \(K_1\sqcup K_1\) if \(x,y\) are false
twins, and by \(K_1\vee K_1\) if they are true twins. Leaving all subsequent
union and join operations unchanged gives exactly \(G\), so \(G\) is a
cograph.
\end{proof}

\section{Cographs generate}

\begin{proposition}\label{prop:cographs-generate}
If \(G\) is a cograph on \(X\), then \(S_G\) diamond-generates \(\B_X\).
\end{proposition}

\begin{proof}
We induct on \(|X|\). The assertion is immediate when \(|X|=1\). Suppose
\(|X|\ge2\). By the recursive definition of a cograph, there is a partition
\(X=X_1\sqcup\cdots\sqcup X_k\), with \(k\ge2\), such that each induced graph
\(G_a=G[X_a]\) is a smaller cograph and \(G\) is either their disjoint union or
their join.

First suppose \(G=G_1\sqcup\cdots\sqcup G_k\). For each \(a\), consider the
Boolean face
\[
F_a^0=\{A\subseteq X:A\subseteq X_a\},
\]
which is naturally isomorphic to \(\B_{X_a}\). If \(i\in X_a\), then
\(N_G(i)=N_{G_a}(i)\), so the seeds of \(S_G\) in directions belonging to
\(X_a\) are precisely the embedded copy of \(S_{G_a}\) in this face. Every
diamond of the face is also a diamond of \(\B_X\). By induction, these seeds
generate every cover edge in \(F_a^0\), and in particular they generate
\(e_i(\varnothing)\) for every \(i\in X_a\). Applying this to all parts
produces the full bottom star \(\botstar\), which generates \(\B_X\) by
\cref{prop:minsize}.

Now suppose \(G=G_1\vee\cdots\vee G_k\). For each \(a\), consider the face
\[
F_a^1=\{(X\setminus X_a)\cup A:A\subseteq X_a\},
\]
again naturally isomorphic to \(\B_{X_a}\). For \(i\in X_a\),
\[
N_G(i)=(X\setminus X_a)\cup N_{G_a}(i),
\]
so the corresponding seeds form the embedded copy of \(S_{G_a}\) in
\(F_a^1\). Induction generates the top star of this face. Its direction-\(i\)
edge is
\[
e_i(X\setminus\{i\}):X\setminus\{i\}\longrightarrow X.
\]
Taking all parts yields the full top star \(\topstar\), which also generates
\(\B_X\).
\end{proof}

\section{The converse}

\begin{lemma}[Induced-subgraph inheritance]\label{lem:induced}
Let \(G\) be a graph on \(X\), and let \(Y\subseteq X\). If \(S_G\)
diamond-generates \(\B_X\), then \(S_{G[Y]}\) diamond-generates \(\B_Y\).
\end{lemma}

\begin{proof}
For every \(i\in Y\),
\[
N_{G[Y]}(i)=N_G(i)\cap Y.
\]
Hence \(\pi_Y(S_G)=S_{G[Y]}\). The result follows immediately from
\cref{lem:projection}.
\end{proof}

\begin{lemma}[Twin-trigger equivalence]\label{lem:twin-trigger}
For distinct vertices \(i,j\) of a graph \(G\), the two seed edges
\(e_i(N_G(i))\) and \(e_j(N_G(j))\) form the lower half of a diamond if and
only if \(i,j\) are false twins. They form the upper half of a diamond if and
only if \(i,j\) are true twins.
\end{lemma}

\begin{proof}
The seed edges form a lower half exactly when their lower endpoints agree:
\[
N_G(i)=N_G(j).
\]
This equality forces \(i\) and \(j\) to be nonadjacent, because otherwise
\(j\in N_G(i)\) but \(j\notin N_G(j)\). It also says that their neighborhoods
outside \(\{i,j\}\) agree. Thus they are false twins. The converse is
immediate from the definition of false twins.

The seed edges form an upper half exactly when
\[
N_G(i)\cup\{i\}=N_G(j)\cup\{j\}.
\]
The membership of \(i\) in the left-hand side forces \(i\in N_G(j)\), and
similarly \(j\in N_G(i)\); hence the vertices are adjacent. After removing
\(i\) and \(j\), the equality gives identical outside neighborhoods, so they
are true twins. Conversely, true twins satisfy the displayed equality.
\end{proof}

\begin{lemma}[A generating neighborhood seed has twins]\label{lem:first-diamond}
If \(|X|\ge2\) and \(S_G\) diamond-generates \(\B_X\), then \(G\) contains a
pair of twins.
\end{lemma}

\begin{proof}
The set \(S_G\) has \(|X|\) edges, whereas
\[
|\E(\B_X)|=|X|2^{|X|-1}>|X|.
\]
Thus the closure sequence starting from \(S_G\) must add an edge. If the first
closure round added nothing, then the sequence would already be stable and no
later round could add anything. Therefore some diamond has its lower or upper
half contained in the original seed set \(S_G\). Its two triggering edges have
distinct directions and are the corresponding two neighborhood seeds. By
\cref{lem:twin-trigger}, their vertices are twins.
\end{proof}

\begin{proposition}\label{prop:converse}
If \(S_G\) diamond-generates \(\B_X\), then \(G\) is a cograph.
\end{proposition}

\begin{proof}
Let \(H\) be any induced subgraph of \(G\) with at least two vertices. By
\cref{lem:induced}, the neighborhood seed \(S_H\) diamond-generates the Boolean
lattice on \(V(H)\). By \cref{lem:first-diamond}, the graph \(H\) contains a
pair of twins. Thus every induced subgraph of \(G\) with at least two vertices
has twins, and \cref{lem:twin-characterization} implies that \(G\) is a
cograph.
\end{proof}

\begin{example}[The first obstruction]\label{ex:p4}
The path \(P_4\) has no pair of twins. By \cref{lem:twin-trigger}, no two edges
of \(S_{P_4}\) form the lower or upper half of a diamond. Therefore its first
closure round adds nothing, so
\[
\dcl{V(P_4)}(S_{P_4})=S_{P_4}.
\]
In particular, \(S_{P_4}\) is not diamond-generating. This is the smallest
example obstructing the cograph condition.
\end{example}

\section{Proof of the classification and consequences}

\begin{proof}[Proof of \cref{thm:main}]
Part \textup{(i)} is \cref{prop:minsize}. For part \textup{(ii)},
\cref{prop:cographs-generate} proves that cographs generate, and
\cref{prop:converse} proves the converse.

For part \textup{(iii)}, let \(S\) be a diamond-generating set of cardinality
\(|X|\). By \cref{cor:encoding}, there is a unique graph \(G\) on \(X\) with
\(S=S_G\). Part \textup{(ii)} implies that \(G\) is a cograph. Conversely,
every cograph \(G\) gives a generating set \(S_G\) with exactly one edge in
each direction and therefore exactly \(|X|\) edges.
\end{proof}

\begin{corollary}\label{cor:counting}
The map \(G\mapsto S_G\) is a bijection from the labeled cographs on \([n]\)
to the minimum-cardinality diamond-generating subsets of \(\E(\B_n)\).
Consequently, these two classes have the same cardinality.
\end{corollary}

\begin{proof}
Surjectivity and the fact that every image is a minimum-cardinality generator
are contained in \cref{thm:main}. Injectivity follows because the open
neighborhoods recover every adjacency of \(G\).
\end{proof}

\section{Further questions}

The classification concerns cardinality, not the number of synchronous closure
rounds. The cograph construction suggests several natural refinements.

\begin{question}
For a cograph \(G\), how can the first round at which
\(\dcl{X}(S_G)=\E(\B_X)\) be read from a cotree of \(G\)? Which cographs
minimize or maximize this closure time among graphs on a fixed number of
vertices?
\end{question}

\begin{question}
What are the inclusion-minimal diamond-generating sets that have more than
\(|X|\) edges?
\end{question}

\begin{question}
Which analogous minimum-generator classifications hold for cover graphs of
other finite distributive or modular lattices?
\end{question}

\section*{Acknowledgments}

The author thanks Vladimir Retakh for helpful discussions and for directing him
to the diamond-closure framework developed with Michael Saks.

\end{document}